\documentclass[12pt]{amsart}
\usepackage[T1]{fontenc}
\usepackage{latexsym} 
\usepackage{amsfonts, amsthm, amsmath,amssymb}
\usepackage{BOONDOX-calo}
\usepackage{geometry}
\usepackage{times}
\usepackage{xcolor}
\usepackage{slashed}
\definecolor{refkey}{rgb}{0,0,1}
\definecolor{labelkey}{rgb}{1,0,0}
\usepackage{graphicx}
\usepackage{float}
\usepackage{mdframed}
\usepackage{ulem}
\usepackage{comment}
\definecolor{darkblue}{rgb}{0.0, 0.0, 0.55}
\definecolor{darkcerulean}{rgb}{0.03, 0.27, 0.49}
\definecolor{darkpowderblue}{rgb}{0.0, 0.2, 0.6}
\definecolor{britishracinggreen}{rgb}{0.0, 0.26, 0.15}

\newenvironment{mgg}{\color{magenta}}{}
\newenvironment{brg}{\color{britishracinggreen}}{}
\newenvironment{red}{\color{red}}{}
\newcommand{\bre}{\begin{red}}
	\newcommand{\ere}{\end{red}}
\newcommand{\bas}{\begin{brg}}
	\newcommand{\eas}{\end{brg}}
\newcommand{\bblu}{}
\newcommand{\eblu}{}
\newcommand{\bmag}{\begin{mgg}}
	\newcommand{\emag}{\end{mgg}}

\NewDocumentCommand{\colorrule}{O{.4pt}m}{{\color{#2}\hrule height#1}\vspace{4mm}}
\newtheorem{thm}{Theorem}[section]
\newtheorem{prop}[thm]{Proposition}

\newtheorem{lem}[thm]{Lemma}

\theoremstyle{definition}
\newtheorem{defn}[thm]{Definition}

\numberwithin{equation}{section}

\def\cH{{\mathcal H}} 
\def\cA{{\mathcal A}}

\def\beq{\begin{equation}} 
	\def\eeq{\end{equation}}

\def\wres{\mathcal{W\!}}
\def\wresd{\mathcal{w}}

\def\IR{\mathbb{R}}

\def\IZ{\mathbb{Z}}

\title[Spectral Metric and Einstein Functionals  for the Hodge-Dirac operator]{Spectral Metric and Einstein Functionals \\[2mm] for the Hodge-Dirac operator} 
\author[L.\ D\k{a}browski]{Ludwik D\k{a}browski${}^{(1)}$}
\address{${}^{(1)}$ SISSA (Scuola Internazionale Superiore di Studi Avanzati), \newline\indent Via Bonomea 265, 34136 Trieste, Italy} 
\email{dabrow@sissa.it} 
\author[A.\ Sitarz]{Andrzej Sitarz${}^{(2)}$}
\author[P.\ Zalecki]{Pawe\l{} Zalecki${}^{(2)}$}
\thanks{This work is supported by the Polish National Science Centre grant 2020/37/B/ST1/01540}
\address{${}^{(2)}$ Institute of Theoretical Physics, Jagiellonian University, \newline\indent
	prof.\ Stanis\l awa \L ojasiewicza 11, 30-348 Krak\'ow, Poland.}
\email{andrzej.sitarz@uj.edu.pl}  
\email{pawel.zalecki@doctoral.uj.edu.pl}
\date{}
\begin{document}
\maketitle
\begin{abstract} \bblu
We examine the metric and Einstein bilinear functionals of differential forms introduced in \cite{DSZ23}, for the Hodge-Dirac operator $d+\delta$ on an oriented, closed, even-dimensional Riemannian manifold. We show that they are equal (up to a numerical
factor) to these functionals  for the canonical Dirac operator on a spin manifold. Furthermore, we demonstrate that the spectral triple for the Hodge-Dirac operator is spectrally closed, which implies that it is torsion-free.\eblu  \\[4mm]
\end{abstract}
	
	
\section{Introduction}
Spectral geometry investigates relationships between geometric structures of manifolds and the spectra of certain differential operators. Its direct and inverse problems are inextricably linked to other areas of mathematics such as number theory, representation theory, and areas of mathematical physics such as quantum mechanics and general relativity. In this regard, starting with the Laplace-Beltrami operator on a closed Riemannian manifold, general Laplace-type operators have been extensively studied, and their spectra provide insights into the geometry and topology of the underlying space. The distribution of eigenvalues, for example, reveals information about the curvature or shape and global geometric properties such as diameter or volume, connectivity, or the presence of holes.
	In this vein, the Dirac-type operators have also been studied, beginning with the canonical Dirac operator on the spin manifold.
	When subsumed into Connes' concept of spectral triples \cite{Co80,Co94}, they "can hear the shape of a drum" \cite{Ka66} in the sense that their equivalence (a suitably strengthened isospectrality) implies the isometricity of manifolds in virtue of the reconstruction theorem \cite{Co13}.
	Furthermore, they allow for broad and captivating generalisations in noncomutative geometry.
	
	Various (interrelated) spectral schemes that generate geometric objects on manifolds such as volume, scalar curvature, and other scalar combinations of curvature tensors and their derivatives are
	the small-time asymptotic expansion of the (localised) trace of the heat kernel \cite{Gi84,Gi04}, certain values or residues of the (localised) zeta function of the Laplacian, the spectral action,
	and the Wodzicki residue $\wres$ (also known as noncommutative residue).
	In this paper, we focus on the latter one, which is the unique (up to multiplication by a constant)
	tracial state on the algebra of pseudo-differential operators
	($\Psi$DO) on a complex vector bundle $E$ over a compact manifold $M$ of dimension $n\geq 2$ \cite{Gu85,Wo87}.
	For the oriented manifold $M$ it is given by an integral formula,
	\begin{equation}
		\wres\,(P) :=
		\int_M \wresd\,(P)
	\end{equation}
	where the density $\wresd\,(P)$ is given in local coordinates by 
	\begin{equation}
 \int_{|\xi|=1} tr\, \sigma_{-n}(P)(x,\xi)~ {\mathcal V}_\xi~ d^n x.
	\end{equation}
Here $tr$ is the trace over endomorphisms of the bundle $E$ at any given point of $M$, 
$\sigma_{-n}(P)$ is the symbol of order $-n$ of a pseudodifferential operator $P$ and ${\mathcal V}_\xi$ denotes the volume form on the unit co-sphere.
	
	When applied to the (scalar) Laplacian $\Delta$ on a Riemannian manifold $M$
	of dimension $n=2m$ equipped with a metric tensor $g$ it yields, 
	in a {\it localized} form,  a functional of $f \in C^\infty(M)$,
	\beq
	\label{WresfL}
	{\mathcal v}(f) :=  \wres\,(f \Delta^{-m})= v_{n-1} \int_M f~vol_g,
	\eeq
	where 
	$$v_{n-1}:=vol(S^{n-1}) = \frac{2\pi^{m}}{\Gamma(m)},$$ 
	is the volume of the unit sphere $S^{n-1}$ in $\IR^n$.
	
	\bblu A startling result regarding a higher power of the Laplacian presented by Connes 
	\cite{Co96} in the early 1990s (see \cite{KaWa95, Ka95} for explicit computations) states 
	that \eblu
	\beq\label{WresfLL}
	{\mathcal R}(f) := \wres\,(f \Delta^{-m+1})= \frac{n-2}{12} v_{n-1} \int_M f R(g) vol_g, 
	\eeq which for $n>2$ and $f=1$ is, up to a constant, a Riemannian analogue of the Einstein-Hilbert action functional of general relativity in vacuum.  Here $R=R(g)$ is the scalar curvature, that is the $g$-trace  
	$R\!=\!g^{jk}R_{jk}$ of the Ricci tensor  with components $R_{jk}$ in local coordinates, where $g^{jk}$ are the raised components of the metric $g$.

	In the noncommutative realm, the spectral-theoretic approach to scalar curvature has been extended to quantum tori in the seminal work of Connes and Tretkoff  
	\cite{CoTr11} and extensively studied by many authors (see references in \cite{DSZ23}).
	
	In the recent paper \cite{DSZ23} we accomplished the task of extracting two other important tensorial geometrical objects through spectral methods. These were the metric tensor ${ g}$ itself, its dual, and the Einstein tensor
	$$ { G}:= \hbox{Ric} - \frac{1}{2} R(g)\, {g},$$
	which directly enters the Einstein field equations with matter and its dual. In fact, for this purpose, we employed the Wodzicki residue of a suitable power of the Laplace type operator or of the Dirac type operator, multiplied by a pair of other differential operators. Notably, we have recovered the tensors ${g}$ and ${G}$ as the density of certain bilinear functionals of vector fields on a manifold $M$, while their dual tensors are the density of bilinear functionals of differential one-forms on $M$. The latter functionals (up to a numerical factor) we have obtained also for the canonical Dirac operators (in case $M$ is a spin manifold).
	Then, using Connes' and Moscovici's \cite{CoMo95} generalisation of pseudodifferential calculus for noncommutative spectral triples, we introduced their conspicuous quantum analogue and probed it on 2 and 4-dimensional noncommutative tori.
	
	The aim of this paper is, employing methods of the Wodzicki residue, to analyse the metric and Einstein functionals for another natural Dirac-type operator, namely the Dirac-Hodge operator $d+\delta$ acting on (complex) differential forms $\Omega(M)$ of arbitrary order on a oriented even-dimensional Riemannian manifold $M$. It is worth mentioning that the associated Hodge-Dirac spectral triple is characterised \cite{DDS18} by the fact that 
\bblu its dense Hilbert subspace of continuous forms provides a Morita equivalence $Cl(M)-Cl(M)$ bimodule,
	where $Cl(M)$ is the $C^*$-algebra of continuous sections of the bundle of Clifford algebras on $M$.
	As is well known, the canonical spectral triple on a spin manifold is instead characterised by the fact that its dense Hilbert subspace of continuous Dirac spinors 
\eblu	
	provides a Morita equivalence $Cl(M)-C(M)$ bimodule, where $C(M)$ is the algebra of continuous complex functions on $M$.
	As our first main result, we demonstrate that these two different pivotal cases yield in fact equal spectral metric and Einstein functionals (up to a numerical factor). Moreover, as our second main result, we prove that the associated spectral triple is spectrally closed, that is, for any operator $T$ of zero-order,
	$$ \wres (T D|D|^{-n}) \equiv 0. $$
	A forthcoming result \cite{DSZ23b} demonstrates that, as a consequence, the Hodge-Dirac operator has no torsion.
	\section{Preliminaries} 
	Let $n=2k$ be the dimension of an oriented, closed, smooth Riemannian manifold $M$.
	We will use capital letters to denote increasing sequences of numbers between $1$ and $n$, of fixed length $0 \leq \ell \leq n$.
	A differential $\ell$-form $\omega = \sum_J \omega_J dx^J$ is determined by its
	coefficients $\omega_J$, with respect to coordinates indicated by the multi-index $J$, where with a slight abuse of notation $0$-forms (i.e. functions) will correspond to $J=\emptyset$.
	
	We introduce the operators $\lambda^j_+$ and $\lambda^j_-$ which respectively raise/lower the degree of forms, with components given by 
	$$ (\lambda^p_+)^I_J = \epsilon^{I}_{pJ}, \qquad (\lambda^p_-)^I_J = \epsilon^{pI}_{J}, $$
	where $\epsilon^{I}_{pJ}=(-)^{|\pi|}$ if the juxtaposed index $pJ$ is a permutation $\pi$ of $I$ and $\epsilon^{I}_{pJ}=0$ otherwise, and similarly for $\epsilon^{pI}_{J}$. They satisfy
	\begin{equation}
		\begin{aligned}
			&\lambda^p_+ \lambda^r_+ + \lambda^r_+ \lambda^p_+ = 0,  \\	
			&\lambda^p_- \lambda^r_- + \lambda^r_- \lambda^p_- = 0,  \\
			&\lambda^p_+ \lambda^r_- +  \lambda^r_- \lambda^p_+= \delta_{pr} \,{\rm id},
		\end{aligned}
	\end{equation}
	which follow from the relations  (c.f. \cite{MMMT})
	\begin{equation}
		\begin{aligned}
			&	\sum_K   \epsilon^{I}_{pK}  \epsilon^{K}_{rJ} =  \epsilon^{I}_{prJ}, \qquad
			&	\sum_K   \epsilon^{pI}_{K}  \epsilon^{rK}_{J} =  \epsilon^{I}_{rpJ}, \\
			&	\sum_K   \epsilon^{I}_{pK}  \epsilon^{rK}_{J} = \delta_{pr} \epsilon^I_J - \epsilon_{pJ}^{rI}, \qquad 
			&	\sum_K   \epsilon^{rI}_{K}  \epsilon^{K}_{pJ} =  \epsilon_{pJ}^{rI}, \\
		\end{aligned}
	\end{equation}
	where the juxtaposed indices can be ordered using a signed permutation.
	We also introduce 
	$$\gamma^p  = -i(\lambda_+^p - \lambda_-^p),$$
	which satisfy the following Clifford algebra relation
	$$\{ \gamma^p, \gamma^r\} =2\delta_{pr}.$$ 
	
	In the rest of the paper, we employ normal coordinates $x$ centred around some fixed point on the manifold. Recall that then the components of the metric tensor $g$, its covariant (raised) components, and the square root of the determinant 
	of the matrix of the components of $g$ and the components of the \bblu Christoffel symbols of the \eblu Levi-Civit\'a connection have  the following Taylor expansion around $x=0$:
	\begin{equation}
		\begin{aligned}	
			&		g_{ab} = \delta_{ab} - \frac{1}{3} R_{acbd} x^c x^d + o({{\bf x}^2}), \\
			&		g^{ab} = \delta_{ab} + \frac{1}{3} R_{acbd} x^c x^d + o({{\bf x}^2}), \\
			&		\sqrt{\hbox{det}(g)}  = 1  - \frac{1}{6} \mathrm{Ric}_{ab} x^a x^b + o({{\bf x}^2}), 	 \\
			&		\Gamma^a_{bc}= -\frac{1}{3}(R_{abcd}+R_{acbd } ) x^{d} +o({{\bf x}^2}).	
		\end{aligned}				
		\label{allNorm} 
	\end{equation}
	Here $R_{acbd}$ and $\mathrm{Ric}_{ab}$ are the components of the Riemann and Ricci tensors, respectively, at the point $x=0$ and we use the notation $o({{\bf x}^k})$ to denote that we expand a function up to the polynomial of order $k$ in the normal coordinates. 
	\bblu The expansion in normal coordinates is more convenient notation to obtain the value
	of relevant quantities (symbols of operators) and their derivatives at a given point on the manifold, which we need to compute the products of symbols according to the rules of
	multiplication of symbols (\ref{composition}).
	The position of indices for $\lambda^p_\pm$, coordinates $x$ and $\xi$, as well as Riemann
	and Ricci tensors at the chosen point is chosen for simplicity only (as we are using normal 
	coordinates and the metric at this point is $\delta_{ab}$). We use Einstein summation convention
	for repeated indices (independently of their position), by $\{\, ,\}$ we denote 
	anticommutators. \eblu  
	\subsection{Hodge-Dirac operator}
	We focus on the Hodge-Dirac operator $D=d+d^*$, where $d$ is the exterior derivative and $d^*$ is its  (formal) adjoint. 
	Using our notation, we compute 
\bblu	(locally) \eblu
	the symbol of $D$, 
	\begin{equation}
		\sigma(D) =  (i  \lambda^p_+ - i g^{pr}  \lambda^r_-) \xi_p
		+ \lambda_-^p \lambda_+^r \lambda_-^s  \Gamma^s_{rt} g^{pt},
		\label{HoDi}
	\end{equation}
	which in normal coordinates takes form 
	\begin{equation}
		\sigma(D) =  - \gamma^p \xi_p  - \frac{1}{3} i \lambda^p_- R_{sapb} x^a x^b \xi_s 
		- \frac{1}{3} \lambda_-^p \lambda_+^r \lambda_-^s 
		(R_{srpa}+R_{spra}) x^a + o({{\bf x}^2}).   
	\end{equation}	
	We compute then the symbols of the Hodge-Dirac Laplacian $D^2$
	in normal coordinates up to orders relevant for our purposes. 
	\begin{lem}
		The three homogeneous symbols of $D^2$ read
		$$ 
		\begin{aligned}
			{\mathfrak a}_2 =& \bigl( \delta_{ab} + \frac{1}{3} R_{acbd} x^c x^d \bigr) 
			\xi_a \xi_b + o({{\bf x}^2}), \\
			{\mathfrak a}_1 =&+ \frac{2}{3} i Ric_{ab} \xi_a x^b 
			- \frac{2}{3} i \lambda_+^p \lambda_-^r   (R_{rpab} + R_{rapb} )  x^b   \xi_a + o({{\bf x}^1}), \\
			{\mathfrak a}_0 = &  +\frac{2}{3}  \lambda_+^p   \lambda_-^r   Ric_{pr}
			+  \frac{1}{3}  \lambda_+^p   \lambda_+^r  \lambda_-^s \lambda_-^t   (R_{tsrp}+R_{trsp })  + o({{\bf x}^0}).
		\end{aligned}
		$$
	\end{lem}
	\begin{proof}
		The computation of the principal symbol ${\mathfrak a}_2$ is obvious for the symbol of order 1:
		$$ 
  	\begin{aligned}
			{\mathfrak a}_1 =& -\frac{1}{3} i  \{\lambda_+^t , \lambda_-^p \lambda_+^r \lambda_-^s  \}
			(R_{srpa}+R_{spra}) x^a \xi_t  
			+ \frac{1}{3} i  \{\lambda_-^t , \lambda_-^p \lambda_+^r \lambda_-^s \}
			(R_{srpa}+R_{spra}) x^a \xi_t  \\ &
			\quad + \frac{1}{3}  \gamma^p \lambda_-^r  (R_{aprb}+R_{abrp}) x^b   \xi_a  \\ &
			= - \frac{1}{3} i \lambda_+^r \lambda_-^s (R_{srta}+R_{stra})  x^a \xi_t  
			- \frac{1}{3} i \lambda_-^p \lambda_+^r  (R_{trpa}+R_{tpra})  x^a \xi_t  \\
			& \quad -  \frac{1}{3} i   \lambda_-^p \lambda_-^s  (R_{stpa}+R_{spta}) x^d \xi_t  
			- \frac{1}{3} i \lambda_+^p \lambda_-^s  (R_{tpsa}+R_{tasp}) x^a   \xi_t  \\
			&  \quad + \frac{1}{3} i \lambda_-^p \lambda_-^s  (R_{tpsa}+R_{tasp}) x^a   \xi_t  \\
			& =  \frac{2}{3} i Ric_{sa}  x^a \xi_s
			- \frac{1}{3} i\lambda_+^r \lambda_-^s 
			(R_{srta} + R_{stra} - R_{trsa} - R_{tsra} +R_{trsa} +R_{tasr})  x^a   \xi_t \\
			& \quad  - \frac{1}{3} i \lambda_-^r \lambda_-^s   
			(R_{stra}+R_{srta} - R_{trsa} -R_{tasr}  ) x^a \xi_t \\
			& =   \frac{2}{3} i Ric_{sa}  x^a \xi_s 
			- \frac{2}{3} i\lambda_+^r \lambda_-^s   (R_{srta} + R_{stra} )  x^a   \xi_t. 
		\end{aligned}
		$$
		and the order $0$ symbol:
		$$ 
		\begin{aligned}
			{\mathfrak a}_0 =
			& -\frac{1}{3} i  \gamma^p   \lambda_-^q \lambda_+^r \lambda_-^s   (R_{sqrp}+R_{srqp })   \\
			=&
			-\frac{1}{3}  \lambda^p_+   \lambda_-^q \lambda_+^r \lambda_-^s   (R_{sqrp}+R_{srqp})\\
			= &  \; \frac{2}{3}  \lambda_+^p   \lambda_-^s   Ric_{ps}
			+  \frac{1}{3}  \lambda_+^p   \lambda_+^r  \lambda_-^q \lambda_-^s   (R_{sqrp}+R_{srqp }) .
		\end{aligned}
		$$
	\end{proof}
\subsection{The inverse of $D^2$ and its powers}
In this section, we  present the results that can be applied to a more general situation than the Hodge-Dirac operator.  \bblu Note, that  since we work with pseudodifferential operators, we
denote by the inverses of elliptic operator the corresponding parametrix. For this reason
we can ignore the kernel of these operators. \eblu
Let us start with the following lemma:
	\begin{lem} \label{inv_laplace_type}
		Let $L$ be a Laplace-type operator with symbol 
		$$
		\sigma(L)= \mathfrak a_2 +\mathfrak a_1 +\mathfrak a_0
		$$
		expressed in normal coordinates as
		$$
		\begin{aligned}
			&\mathfrak a_2 = \bigl(\delta_{ab} +\frac13R_{acbd}x^cx^d\bigr) \xi_a\xi_b +o({{\bf x}^2}),\\
			&\mathfrak a_1 = i P_{ab}\xi_ax^b + o({{\bf x}^1}),\\
			&\mathfrak a_0 = Q + o({{\bf x}^0}).
		\end{aligned}
		$$
		Then the theree leading symbols of $\sigma(L^{-k})= \mathfrak c_{2k} +\mathfrak c_{2k+1} +\mathfrak c_{2k+2}$ are:
		$$
		\begin{aligned}
			& \mathfrak c_{2k}= ||\xi||^{-2k-2} \left( \delta_{ab} - \frac{k}{3} R_{acbd} x^c x^d \right)  \xi_a \xi_b + o({{\bf x}^2}),\\
			&\mathfrak c_{2k+1}= -ik ||\xi||^{-2k-2} P_{ab}\xi_ax^b + o({{\bf x}^1}),\\
			& \mathfrak c_{2k+2}= -k||\xi||^{-2k-2} Q +k(k+1) ||\xi||^{-2k-4}\left(P_{ab}-\frac13 Ric_{ab}\right) \xi_a \xi_b + o({{\bf x}^0}).
		\end{aligned}
		$$
	\end{lem}
	\begin{proof}
\bblu First, observe that we follow the notation of most papers (see \cite{CoTr11})	 
and to simplify the notation indicate a negative order $-k$ by $k\geq 0$, so $\mathfrak b_2$
below denotes the symbol of order $-2$. \eblu 
We start with computing leading symbols of the inverse of $L$, i.e. $\sigma(L^{-1})= \mathfrak b_2 +\mathfrak b_3 +\mathfrak b_4$ using the fact, that $\sigma(LL^{-1})=\sigma(1)=1$. We have:
$$
		\begin{aligned}
			\mathfrak b_2 = &\; (\mathfrak a_2)^{-1} = ||\xi||^{-4} \left(\delta_{ab} - \frac13R_{acbd}x^cx^d\right) \xi_a\xi_b +o({{\bf x}^2}),\\
			\mathfrak b_3 =& \;\mathfrak b_2(-\mathfrak a_1 \mathfrak b_2 +i \partial_\xi^a \mathfrak a_2 \partial_a \mathfrak b_2)= -i ||\xi||^{-4} P_{ab}\xi_ax^b + o({{\bf x}^1}),\\
			\mathfrak b_4 =& \;\mathfrak b_2\left(-\mathfrak a_0 \mathfrak b_2 - \mathfrak a_1 \mathfrak b_3 + i \partial_\xi^a \mathfrak a_2 \partial_a \mathfrak b_3+ i \partial_\xi^a \mathfrak a_1 \partial_a \mathfrak b_2 + \frac{1}{2} \partial_\xi^{ab} \mathfrak a_2 \partial_{ab} \mathfrak b_2\right)\\
			=& -||\xi||^{-4} Q +2 ||\xi||^{-6}\left(P_{ab}-\frac13 Ric_{ab}\right) \xi_a \xi_b + o({{\bf x}^0}).
		\end{aligned}
$$
	To finish we apply Lemma\,A1 in \cite{DSZ23} and compute the three leading symbols of the powers of the pseudodifferential operator $L^{-k}$.
\end{proof}
Using the above  lemma for $L=D^2$, with the Hodge-Dirac operator $D$ \eqref{HoDi} , we get the following result.
	\begin{prop}\label{prop23}
The leading symbols of  $D^{-2k}$ are, up to the appropriate order in $x$,
		\begin{equation} 
			\begin{aligned}
				&\mathfrak c_{2k} = ||\xi||^{-2k-2} \left( \delta_{ab} - \frac{k}{3} R_{acbc} x^c x^d \right)  \xi_a \xi_b + o({\bf x}^2),\\
				&\mathfrak c_{2k+1}=-\frac{2}{3} ki||\xi||^{-2k-2}
				\mathrm{Ric}_{ab} x^b \xi_a 
				{
					+} \frac{2}{3} k i ||\xi||^{-2k-2}  \lambda_+^r \lambda_-^s   \bigl(  R_{srba} + R_{sbra}  \bigr) x^a   \xi_b  + o({\bf x}^1)  \\
				&\mathfrak c_{2k+2}=\frac{k(k+1)}{3} ||\xi||^{-2k-4} \mathrm{Ric}_{ab}\xi_a\xi_b \\
				& \qquad \qquad  
				- \frac{2}{3} k (k+1) ||\xi||^{-2k-4}  \, 
				\lambda_+^r \lambda_-^s   (R_{srab} + R_{sarb} )  \xi_a   \xi_b   \\
				& \qquad \qquad 
				+\frac{1}{3} k ||\xi||^{-2k-2}\lambda_+^p  \lambda_-^q  \lambda_+^r \lambda_-^s   (R_{sqrp}+R_{srqp }) 
				+ o({{\bf x}^0}).  
			\end{aligned}\label{LapTF3}
		\end{equation}	
	\end{prop}
\begin{proof}
	For $L=D^2$ we substitute in Lemma \ref{inv_laplace_type}
		$$
		\begin{aligned}
			&P_{ab} = \frac{2}{3} Ric_{ab} 
			- \frac{2}{3} \lambda_+^r \lambda_-^s   (R_{srab} + R_{sarb} ),  \\
			&Q= -  \frac{1}{3}  \lambda_+^p  \lambda_-^q \lambda_+^r \lambda_-^s   (R_{sqrp}+R_{srqp }).
		\end{aligned}
		$$
	\end{proof}
	\section{Spectral functionals}
In \cite{DSZ23}, we defined two spectral functionals for finitely summable spectral triples, which for the canonical spectral triple over the spin manifold $M$ allow to recover the metric and the Einstein tensors, viewed as bilinear functionals over a pair of one-forms. We recall the definition:
	\begin{defn}[cf. \cite{DSZ23}, Definition 5.4]\label{Dwres}
		If $(\cA, D, \cH)$ is a $n$-summable spectral triple, let $\Omega^1_D$ be the $\cA$ bimodule of one forms generated by  $\cA$ and $[D, \cA]$. Moreover, assume there exists a generalised algebra of pseudodifferential operators which contains $\cA$, $D$,  $|D|^\ell$ for $\ell\in\IZ$ with
		a tracial state $\wres$ over this algebra (called a noncommutative residue), which identically 
		vanishes  on $T |D|^{-k}$ for any $k>n$ and a zero-order operator $T$ (an 
		operator in the algebra generated by $\cA$ and $\Omega^1(\cA)$). Then, 
\bblu 		
		for $u,w \in \Omega^1_D(\cA)$, we call 
		$$ \mathcal{g}_D(u,w) := \wres (u w|D|^{-n}), $$	
		{\em metric functional}, and
		$$ \mathcal{G}_D(u,w) := \wres (u \{ D, w \} D |D|^{-n}). $$	
		{\em Einstein functional}.
\eblu
	\end{defn}

\subsection{Hodge-de\,Rham spectral triple}

We compute now these functionals for 
$\cA=C^\infty(M)$ and $D=d+\delta$, identifying for 
$\dim M = n \geq 2$ ,		
 $$\Omega^1_D(A)\simeq\Omega^1(M)$$ 
\bblu
via the correspondence of local components
$$	 \; u=\gamma^pu_p \;\; \leftrightarrow \;\;  U=u_pe^p,
$$
with respect to an orthonormal  coframe $e^p$. \eblu First, let us compute  the metric functional ${\mathcal g}$. 
	
	\begin{prop}
		For $U, V\in \Omega^1(M)$, the metric spectral functional reads
			\begin{equation}
				{\mathcal g}(U,W) = 2^n v_{n-1}\int\limits_M  
g(U,W)\,vol_g.				
			\end{equation}
		\end{prop}	
\bblu		
\begin{proof}
It can be seen that it is adequate to expand $U, W$, and so $u, w$, up to $o({{\bf x}^0})$ in normal coordinates: 
$$  u =\gamma^p u_p + o({{\bf x}^0}), \qquad  w =\gamma^r w_r+ o({{\bf x}^0}).$$
We compute (locally) the density  of  ${\mathcal W}\left( UW |D|^{-n}  \right)$ as
$$
\int_{||\xi||=1} \hbox{Tr\ } (\gamma^p \gamma^r u_p w_r )  {\mathfrak c}_{n}(D)\,d^nx
			 = v_{n-1}   
\sqrt{g}\,  \hbox{Tr\ } (\gamma^p \gamma^r u_p w_r )\,d^nx  =  
2^n v_{n-1}
\sqrt{g} \,u_p w_p \, d^nx.
$$		
The factor $2^n$ comes from the trace of $1$ over the space of differential forms. 
\end{proof}
Next, we have,	
			\begin{prop}\label{HDEin}
For $U, V\in \Omega^1(M)$, the Einstein functional reads 
				$$
				\mathcal{G}(U,W)=\frac{2^n}{6}v_{n-1}\int\limits_M 
G(U,W) vol_g,
				$$    
where $G$ is the Einstein tensor for $M$.
\end{prop}
\eblu
Before we begin with the proof  let us demonstrate some useful lemmas. The first computes $\mathcal W(ED^{-2m+2})$ for two specific cases of endomorphism $E$. 
\begin{lem} \label{funkcjonaly}
\bblu If $E$ is locally given by 
				$$e^{(0)} +e^{(2)}_{pq}\gamma^p\gamma^q,$$
the density of the functional $	\mathcal W(ED^{-2m+2})$ reads locally
				$$	\frac{n-2}{24}2^nv_{n-1}
		\sqrt{g} R( -e^{(0)} -e^{(2)}_{pp} )\,d^nx.
				$$
On the other hand, if $\tilde E$ is locally given by 
$$\tilde e^{(0)} +\tilde e^{(2)}_{pq}\lambda_+^p\lambda_-^q,  $$
the local density of the functional $\mathcal W(\tilde E D^{-2m+2})$ reads locally
		$$-\frac{n-2}{24}2^nv_{n-1}
		\sqrt{g} R\tilde e^{(0)} \,d^nx. $$ \eblu			
			\end{lem}

The proof is based on direct calculations using Proposition \ref{prop23}:
\begin{equation}
		\begin{aligned}
	&\hbox{Tr\ } \left(E\int_{||\xi||=1} \mathfrak c_{2(m-1)+2}\right)
					{\mathcal V}_\xi\, d^nx=\\
					&\frac{n-2}{12}v_{n-1} \hbox{Tr\ }\Big(E\left(2(R_{srqp}+R_{sqrp})\lambda_+^p\lambda_-^q\lambda_+^r\lambda_-^s +R-2Ric_{qp}\lambda_+^p\lambda_-^q\right)\Big)d^nx
				\end{aligned}
			\end{equation}
Next, more generally, we state:
\begin{lem} \bblu
For $P$ such that $\sigma(P)=F^{ab}\xi_a\xi_b+G^a\xi_a+H$, where locally 
				$$ F^{ab}= f^{(0)ab} +f^{(2)ab}_{pq}\gamma^p\gamma^q, $$
$\mathcal w(PD^{-n})$ reads locally
			$$	 v_{n-1} 
		\left(	\hbox{Tr\ }H + \frac{2^n}{24} R(-f^{(0)aa}- f^{(2)aa}_{pp})\right) d^nx.$$
Also, for $\tilde P$ such that $\sigma(\tilde P)=\tilde F^{ab}\xi_a\xi_b+G^a\xi_a+H$, where locally 
				$$ \tilde F^{ab} = \tilde f^{(0)ab} +\tilde f^{(2)ab}_{pq}\lambda_+^p\lambda_-^q,$$
$\mathcal w(\tilde P D^{-n})$ reads locally
				$$
	v_{n-1}	\left(		\hbox{Tr\ }H + \frac{2^n}{48}
				\left( -2R \tilde f^{(0)aa} -\tilde f^{(2)aa}_{pp} R+ 2(\tilde f^{(2)ab}_{pq}+ \tilde f^{(2)ba}_{pq} )R_{paqb} \right)\right)d^nx. 
				$$ \eblu 
			\end{lem}
			The proof follows directly by computation using Proposition \ref{prop23} and 
			Lemma \ref{kawalki} applied to the explicit expression:
			\begin{equation}\label{secondfunctional}
				\begin{aligned}
			\int_{||\xi||=1}\sigma_{-2m}(PD^{-2m}){\mathcal V}_\xi\,
		&= \frac{1}{6}F^{aa} \left[(R_{srqp}+R_{bqrs}) \lambda_+^p\lambda_-^q \lambda_+^r\lambda_-^s 
					+\frac12 R - R_{pq}\lambda_+^p\lambda_-^q\right]\\
					&+\frac{1}{6}F^{ab}[-R_{ab}+(R_{qapb}+R_{paqb})\lambda_+^p\lambda_-^q]    
					+H.
				\end{aligned}
			\end{equation}
			\begin{proof}[Proof of Proposition \ref{HDEin}]
				We begin with computing the  symbol of  \bblu $ u D  w D $ at $x=0$, \eblu
				where it suffices to expand 				$u$ and $w$ as:
				$$ u =  \gamma^p {u}_p + o({{\bf x}^0}), \qquad  
				 w =  \gamma^s w_s + \gamma^s w_{sa}  x^a + o({{\bf x}^1}),$$
				and thus locally:
				$$ 
				\begin{aligned}
					u D wD  = & 
					\gamma^p \gamma^q \gamma^r \gamma^s {u}_p w_r \xi_q \xi_s 
					- i  \gamma^p \gamma^q \gamma^r \gamma^s {u}_p w_{rq}  \xi_s \\
					&  -\frac{1}{3} i \gamma^p \gamma^q \gamma^r 
					\lambda_-^s \lambda_+^t \lambda_-^z  (R_{ztsq} +R_{zstq})  u_p w_{r} +o({{\bf x}^0}).
				\end{aligned}
				$$
				Then, we use lemma \ref{funkcjonaly} for $P= uD wD$ and $E= uw$. In this case we have
				$$
				\begin{aligned}
					&E=u_p w_q\gamma^p\gamma^q, \\
					& H  = -\frac{1}{3} i \gamma^p \gamma^q \gamma^r 
					\lambda_-^s \lambda_+^t \lambda_-^z  (R_{ztsq} +R_{zstq})  u_p w_{r}\\
					& \quad=-\frac{2}{3}i\gamma^p 
					\lambda_-^q \lambda_+^r \lambda_-^s  (R_{srqt} +R_{sqrt})  u_p w_{t} 
					+ \frac{1}{3} i \gamma^p \gamma^q \gamma^r 
					\lambda_-^s \lambda_+^t \lambda_-^z  (R_{ztsr} +R_{zstr})  u_p w_{q}, \\ 
					&F^{pq}\xi_p\xi_q=  \gamma^r \gamma^p \gamma^s \gamma^q u_r w_s \xi_p\xi_q= (2u_r w_p \delta_{qs}\gamma^r\gamma^s 
					- u_r w_s \gamma^r\gamma^s \delta_{pq})\xi_p\xi_q,
				\end{aligned}
				$$
				where we used that $\gamma^p\gamma^q\xi_p\xi_q = \delta_{pq}\xi_p\xi_q$. 
				Next, we see,
				$$
				\begin{aligned} 
					&e^{(0)}=0, \qquad && e^{(2)}_{ab}=u_a w_b, \\
					& f^{(0)}=0, \qquad && f^{(2)ab}_{cd}=2u_c w_a\delta_{bd}
					-u_c w_d\delta_{ab}.
				\end{aligned}
				$$
				Finally, the contribution arising from $E$ gives:
				$$
				-\frac{n-2}{24}2^nv_{n-1}R\, u_aw_a.
				$$
				whereas the part from $F$ is,
				$$
				-\frac{2^n}{24} v_{n-1} R f^{(2)aa}_{ii}= 
				\frac{n-2}{24} 2^n v_{n-1} R\, u_a w_a.
				$$
			These two terms cancel each other, and we are left with terms that arise from  $H$. The only possible terms in $\hbox{Tr\ }H$ are linear combinations of $u_aw_aR$ and $u_aw_bRic_{ab}$, thus, we know, that the result is symmetric in $u_a,w_b$. It allows us to simplify the second term in $H$:
				$$
				H=-\frac{2}{3} i \gamma^p 
				\lambda_-^q \lambda_+^s \lambda_-^t  (R_{tsqb} +R_{tqsb})  u_p w_{b} + \frac{1}{3} i  \gamma^r 
				\lambda_-^q \lambda_+^s \lambda_-^t  (R_{tsqr} +R_{tqsr})  u_a w_{a} +\dots,
				$$
				where "$\dots$" are terms antisymmetric in $u_a,w_b$, which we can neglect. We can also insert $-i\lambda_+$ instead of the remaining $\gamma$'s because the part with $\lambda_-$ will be traceless. Now, using the lemma \ref{slady_naprzemienne} we get:
				$$
				\hbox{Tr\ }H= \frac{1}{6} \hbox{Ric}_{ab}\, u_a w_b  -\frac{1}{12} R \,u_a w_a= \frac{1}{6} G_{ab}\, u_a w_b .
				$$
				This proves the result.
			\end{proof}
\bblu
We deduce that for the spin$_c$ manifolds the spectral functionals for the Hodge-de\,Rham spectral triple are equal, up to the rank of the vector bundles, to those for the canonical spin$_c$ spectral triple.
\eblu			
			\subsection{Spectral closedness and torsion}
			In this section, we will prove that the Hodge-Dirac spectral triple has the property of being spectrally closed.
			\begin{thm}
				Let $T$ be an operator of order $0$ from the algebra generated by $a [D,b]$, $a,b \in C^\infty(M)$.
				Then,
				$$ \wres (T D |D|^n) = 0. $$
			\end{thm}
			\begin{proof}
				If we compute the symbol of  $TD$ at a chosen point on the manifold $M$ in normal ccordinates at
				$x=0$ we obtain,
				$$ \sigma(T D)  = T (- \gamma^p \xi_p). $$
				Next, if we combine it with Proposition \ref{prop23} we see that the symbol of order $-n$
				of $T D |D|^n$ is:
				$$ \sigma_{-n}(T D |D|^{-n}) = 0 + o(\mathbf{x^0}). $$
				This ends the proof.
			\end{proof}
		
	\bblu
		As a consequence, we demonstrate in \cite{DSZ23b} that the Hodge-Dirac spectral triple is torsion-free. It is interesting to study generalised Hodge-Dirac operators, which are defined through an arbitrary linear connection, are not metric compatible, and have a torsion. Also,
		the extension to noncommutative Hodge-de Rham calculi and a comparison with the approach of \cite{FGK16} would be of great interest. Note that we focused here on even-dimensional manifolds only for the sake of simplicity and the result  easily extends to the case of odd-dimensional manifolds (for technical details that include the computation of 
		symbols, compare \cite{DSZ23b}). 
\eblu			
			\appendix
			\section{Details of computations}
			\bblu
			We begin with a formula for the product of two pseudodifferential operators, $P$ and $Q$,
			which have expansion in homogeneous symbols,
			\begin{equation}
				\sigma (P)(x,\xi )=\sum \limits _{\alpha} \sigma (P)_{\alpha}(x, \xi),
				\hspace{20pt}
				\sigma (Q)(x,\xi )=\sum \limits _{\beta} \sigma (Q)_{\beta}(x, \xi) ,
				\label{eqA.1}
			\end{equation}
			respectively, where $\alpha , \beta $ are multiindices (such that $|\alpha| = \sum_k \alpha^k$ and $|\beta|$ are bounded from above). The symbol $\sigma (P)_{\alpha}(x, \xi)$ is
			homogeneous of order $|\alpha|$ if  for $r>0$ $\sigma (P)_{\alpha}(x, r\xi) = r^{|\alpha|} 
			\sigma (P)_{\alpha}(x, \xi)$.
			The composition
			rule for the symbols of their product takes the form \cite{Gi84}.
			%
			\begin{equation}
				\sigma (PQ)(x,\xi )=\sum \limits _{\beta}
				\frac{(-i)^{|\beta |}}{|\beta |!}\partial ^{\xi}_{\beta }\sigma (P)(x,
				\xi )\partial _{\beta }\sigma (Q)(x,\xi ),
				\label{composition}
			\end{equation}
			where $\partial _{a}^{\xi}$ denotes the partial derivative with respect
			to the coordinate of the cotangent bundle.
			\eblu

			\begin{lem}\label{Tra1}
				A direct computation of traces of products of $\lambda$ matrices is based on the following
				recursive formula: 
				$$
				\begin{aligned}
					\mathrm{Tr\ } &\lambda^{p_1}_+  \cdots  \lambda^{p_k}_+	\lambda^{q_1}_-  \cdots  \lambda^{q_k}_-  \\
					&=  \frac{1}{2}  \sum\limits_{j=1}^k (-1)^{k-j} \delta^{p_1 q_j}\; \mathrm{Tr\ } \,
					\bigl( \lambda^{p_2}_+  \cdots  \lambda^{p_k}_+	\lambda^{q_1}_-  \cdots    \lambda^{q_{j-1}}_-  \lambda^{q_{j+1}}_- \cdots \lambda^{q_k}_-  \bigr).
				\end{aligned}
				$$
				In particular, we have
				\begin{equation}
					\label{slady_naprzemienne}
					\begin{aligned}
						&\mathrm{Tr\ }( \lambda_+^p \lambda_-^q ) =
						2^{n-1} \delta^{pq}, 
						\\
						&\mathrm{Tr\ }( \lambda_+^{p_1} \lambda_-^{q_1}\lambda_+^{p_2} \lambda_-^{q_2} ) =
						2^{n-2}(\delta^{p_1q_1} \delta^{p_2q_2} + \delta^{p_1q_2} \delta^{p_2q_1}),
					\end{aligned}
				\end{equation}
				and
				\begin{equation}
					\begin{aligned}	
						\mathrm{Tr\ }( \lambda_+^{p_1} &\lambda_-^{q_1}\lambda_+^{p_2} \lambda_-^{q_2} \lambda_+^{p_3} \lambda_-^{q_3} ) = \\
						&=	2^n \bigl( \frac18 \sum_{\sigma\in S_3} \delta^{p_1q_{\sigma(1)}} \delta^{p_2q_{\sigma(2)}} \delta^{p_3q_{\sigma(3)}} 
						\!-\! \frac14 \delta^{p_1q_{2}} \delta^{p_2q_{3}} \delta^{p_3q_{1}} \bigr).	
					\end{aligned}	
				\end{equation}
			\end{lem}
			Next, we present the results on traces of products of $\gamma$ and $\lambda$ matrices.
			\begin{lem}
				\begin{equation}\label{slady_gamma_lambda}
					\mathrm{Tr\ }(\gamma^p\gamma^q\lambda_+^r\lambda^s_-)=
					2^{n-2}\bigl(
					2\delta^{pq}\delta^{rs}
					+\delta^{ps}\delta^{qr}
					-\delta^{pr}\delta^{qs} \bigr)
					= 2^{n-1} \bigl(\delta^{pq}\delta^{rs}
					- \frac{1}{2} \varepsilon^{pq}_{rs} \bigr),
				\end{equation}
				\begin{equation}\label{corollary2}
					\mathrm{Tr\ }(\gamma^p\gamma^q\lambda_+^r
					\lambda_-^s\lambda_+^t\lambda_-^z) =
					2^{n-2} \delta^{pq}(\delta^{rs}\delta^{tz}+
					\delta^{rz}\delta^{st})
					-2^{n-3}(\delta^{rs}\varepsilon^{pq}_{tz}
					+\delta^{st}\varepsilon^{pq}_{rz}
					+\delta^{tz}\varepsilon^{pq}_{rs}
					+\delta^{rz}\varepsilon^{pq}_{st})
				\end{equation}	
			\end{lem}
			We skip the computational proof, which is based on expressing  $\gamma$-matrices in terms of $\lambda$-matrices, 
			\begin{equation*}
				\gamma^p\gamma^q= -\lambda_+^q\lambda_-^p + \lambda_+^p\lambda_-^q+\delta^{pq}+\dots,
			\end{equation*}
			and using the results of Lemma \ref{Tra1}.
			
As a consequence, we obtain the following identities for the geometric quantities:
			\begin{lem}\label{kawalki}
					In normal coordinates  around $x=0$ we have the following identities		
					\begin{equation}
						\begin{aligned}
							&	\mathrm{Tr\ }( \lambda_+^{p} \lambda_-^{q} ) Ric_{pq}=
							2^{n-1} R , \\
							& 	\mathrm{Tr\ }( \lambda_+^{p} \lambda_-^{q} )(R_{paqb}+R_{qapb})=
							2^n \mathrm{Ric}_{ab} ,\\
							& 	\mathrm{Tr\ }(\lambda_+^{p} \lambda_-^{q} \lambda_+^{r} \lambda_-^{s} )(R_{srqp}+R_{sqrp})=
							-2^{n-2} R ,\\
							& 	\mathrm{Tr\ }( \lambda_+^{p} \lambda_-^{q}\lambda_+^{r} \lambda_-^{s} )\mathrm{Ric}_{rs}=
							2^{n-2} (\delta_{pq}R + \mathrm{Ric}_{pq}) ,\\
							& 	\mathrm{Tr\ }( \gamma^{p} \gamma^{q}\lambda_+^{r} \lambda_-^{s} )
							\mathrm{Ric}_{rs}=
							2^{n-1} \delta_{pq}R ,\\
							& 	\mathrm{Tr\ }( \lambda_+^{p} \lambda_-^{q}\lambda_+^{r} \lambda_-^{s} )(R_{rasb}+R_{sarb})=
							2^{n-2} (2\delta_{pq}Ric_{ab} + R_{qapb} + R_{paqb}) ,\\
							& 	\mathrm{Tr\ }( \gamma^{p} \gamma^{q}\lambda_+^{r} \lambda_-^{s} )(R_{rasb}+R_{sarb})=
							2^{n} \delta_{pq}\mathrm{Ric}_{ab} ,\\
							& 	\mathrm{Tr\ }( \lambda_+^{p} \lambda_-^{q} \lambda_+^{r} \lambda_-^{s} \lambda_+^{t} \lambda_-^{z} )(R_{ztsr}+R_{zstr})=
							2^{n-3} (-R\delta_{pq}+2\mathrm{Ric}_{pq}) ,\\
							& 	\mathrm{Tr\ }( \gamma^{p} \gamma^{q}\lambda_+^{r} \lambda_-^{s} \lambda_+^{t} \lambda_-^{z} )(R_{ztsr}+R_{zstr})=
							-2^{n-2}\delta_{pq}R. 
						\end{aligned}
					\end{equation}		
				\end{lem}   
				\begin{proof}
					Direct computation using  \eqref{slady_naprzemienne}-\eqref{slady_gamma_lambda} and the
					properties of the Riemann and Ricci tensors.
				\end{proof}
				
				\vspace{3mm}
				\section*{\bf Acknowledgements}
				LD acknowledges that this work is part of the project Graph Algebras partially supported by EU grant HORIZON-MSCA-SE-2021 Project 101086394 and
				was partially supported the University of Warsaw Thematic Research Programme "Quantum Symmetries".
				
				AS and PZ acknowledge that this work is supported by the Polish National Science Centre grant 2020/37/B/ST1/01540.

			\end{document}